\documentclass[11pt,a4paper]{amsart}
\usepackage{amsmath,amssymb,amsthm,hyperref}
\usepackage{thmtools}

\usepackage[style=alphabetic,sorting=nty,doi=false,isbn=false,url=false]{biblatex}

\usepackage{cleveref}

\numberwithin{equation}{section}
\hypersetup{  pdfborder={0 0 0},
  colorlinks   = true,
  urlcolor     = blue,
  linkcolor    = blue,
  citecolor   = red}

\hypersetup{hypertexnames=false}

\newcommand{\N}{\mathbb{N}}
\newcommand{\Z}{\mathbb{Z}}
\newcommand{\Q}{\mathbb{Q}}

\newcommand{\cM}{\mathcal{M}}

\newcommand{\cG}{\mathcal{G}}
\newcommand{\tG}{\widetilde{G}}

\newcommand{\gbs}{\mathrm{GBS}}

\renewcommand{\d}{\mathrm{d}}

\newcommand{\vr}{\leqslant_{vr}}

\DeclareMathOperator{\rank}{rank}
\DeclareMathOperator{\st}{Stab}

\DeclareMathOperator{\cent}{C}
\DeclareMathOperator{\out}{Out}
\DeclareMathOperator{\aut}{Aut}
\DeclareMathOperator{\inn}{Inn}

\numberwithin{equation}{section}

\newtheorem{thm}{Theorem}[section]

\newtheorem{prop}[thm]{Proposition}
\newtheorem{lemma}[thm]{Lemma}

\newtheorem{cor}[thm]{Corollary}

\theoremstyle{definition}

\newtheorem{conj}[thm]{Conjecture}

\theoremstyle{remark}

\newtheorem{rem}[thm]{Remark}

\title[Property (LR) for virtually free groups]{Property (LR) and an embedding theorem for virtually free groups}

\author{Ashot Minasyan}
\address{CGTA, School of Mathematical Sciences, University of Southampton, Highfield, Southampton, SO17~1BJ, United Kingdom}
\email{aminasyan@gmail.com}


\begin{document}

\begin{abstract} We prove that every virtually free group $G$ has property (LR) of Long and Reid: each finitely generated subgroup of $G$ is a retract of a finite index subgroup. The main ingredient in the proof is a new embedding result stating that every countable virtually free group embeds in a double of a finite group. 

As a corollary, we show that any group commensurable with the direct product of a free group and a finitely generated abelian group has (LR). This applies to 
generalized Baumslag-Solitar groups of arbitrary rank $n \in \N$ with finite monodromy, which, in particular, include all non-cyclic one-relator groups with center.
\end{abstract}

\maketitle
\section{Introduction}
A famous theorem of M. Hall \cite{M_Hall} implies that every finitely generated subgroup of a free group is a free factor of a finite index subgroup. Conversely, the work of Brunner and Burns \cite{Brun-Burns}, combined with the accessibility of finitely presented groups proved by Dunwoody \cite{Dunw-acc}, tells us that every finitely presented group satisfying the latter \emph{M. Hall's property} must be virtually free (i.e., it has a free subgroup of finite index). However, not all virtually free groups have this property, because finite order elements can have infinite centralizers and free factors are always malnormal (see \cite{Min-virt_retr_props} for a characterization of virtually free groups with M. Hall's property). 

The best analogue of M. Hall's theorem that one could hope to hold in all virtually free groups is \emph{property (LR)}, defined by Long and Reid in \cite{Long-Reid}, stating that every finitely generated subgroup is a retract of some subgroup of finite index. Recall that a subgroup $H$ is a \emph{retract} of a group $K$ if there is a homomorphism $\rho:K \to H$ such that $\rho(h)=h$, for all $h \in H$ (equivalently, $K=N \rtimes H$, for some $N \lhd K$). A subgroup $H$ is a \emph{virtual retract} of a group $G$ if $H$ is a retract of a finite index subgroup of $G$. Thus $G$ has (LR) if every finitely generated subgroup is a virtual retract.

The main result of this note is the following theorem, answering questions of the author from \cite[Question~11.1]{Min-virt_retr_props} and \cite[Question~20.59]{Kourovka}.

\begin{thm} \label{thm:virt_free->(LR)} Virtually free groups have property (LR).
\end{thm}

Of course, M. Hall's theorem implies that free groups have (LR), but this property is quite sensitive and is not always inherited by finite index supergroups (see \cite[Proposition~1.7]{Min-virt_retr_props}). The only other known classes of groups with (LR) are finitely generated virtually abelian groups \cite{Min-virt_retr_props}, fundamental groups of closed surfaces \cite{Scott-surface}, limit groups \cite{Wilton-limit} and some lamplighter groups \cite{Min-virt_retr_props}; additionally, property (LR) is preserved by free products \cite{GMS} (but not by direct products).

Every group with (LR) is \emph{subgroup separable} (a.k.a. LERF), i.e., finitely generated subgroups are closed in the profinite topology, see \cite{Long-Reid,Min-virt_retr_props}, but property (LR) is much more restrictive. For example,  every virtually polycyclic group is subgroup separable but it does not have (LR) unless it is virtually abelian, see \cite[Proposition~9.1]{Min-virt_retr_props}. 
However, (LR) is preserved under passing to arbitrary subgroups, so in order to prove \Cref{thm:virt_free->(LR)} for  a finitely generated virtually free group, we first embed it in a virtually free group of simpler form.

Let $C$ be a group with a subgroup $B \leqslant C$. Recall that a \emph{double of $C$ over $B$} is the amalgamated product
\begin{equation}\label{eq:double_of_C_over_B}
G=C*_B C'=\langle C,C' \mid b=\xi(b),~\text{ for all } b \in B\rangle,
\end{equation}
where $C'$ is a copy of $C$ and $\xi:C \to C'$ is an isomorphism. If the amalgamated subgroup $B$ is not important, we will simply say that $G$ is a \emph{double of $C$}.
The following, perhaps unexpected, embedding result is key to our proof of \Cref{thm:virt_free->(LR)}, and may also be of independent interest.

\begin{thm}\label{thm:main_emb} Every countable virtually free group embeds as a subgroup in a double of a finite group.
\end{thm}

We actually prove a stronger version of \Cref{thm:main_emb} in the case when the virtually free group is finitely generated, see \Cref{thm:spec_HNN_embeds_in_double}. In particular, we obtain the following consequence.

\begin{cor}\label{cor:free-by-p_in_double_of_p} If $A$ is a finite $p$-group, for some prime $p \in \N$, (or a finite solvable group) then every finitely generated free-by-$A$ group embeds in a double of a finite $p$-group (respectively, a finite solvable group). 
\end{cor}

Since a double of a finite group is virtually free,  we can combine \Cref{thm:virt_free->(LR)} with \Cref{thm:main_emb} to deduce the following.

\begin{cor} Every finitely generated virtually free group embeds as a virtual retract in a double of a finite group.
\end{cor}

In order to prove \Cref{thm:main_emb}, we can assume that the virtually free group is finitely generated (by Scott's embedding theorem \cite{Scott-emb}). By a theorem of Karrass, Pietrowski and Solitar \cite{KPS}, building on Stallings' theory of groups with infinitely many ends \cite{Stal},  each finitely generated virtually free group  splits as the fundamental group of a finite graph of groups with finite vertex groups. Using this decomposition, in a series of lemmas we show that every finitely generated virtually free group embeds in a single special HNN-extension of a finite group $A'$ (\Cref{prop:emb_in_single_spec_HNN}). Finally, in \Cref{thm:spec_HNN_embeds_in_double} we embed this special HNN-extension in a double of $A' \times C_p$, for any prime $p \in \N$.

We have thus reduced the proof of \Cref{thm:virt_free->(LR)}, in the case when the virtually free group is countable, to showing (LR) for a double $G$, given by  \eqref{eq:double_of_C_over_B} for some finite group $C$. We deal with this case using an idea of Scott from \cite{Scott-surface} (and also Haglund's interpretation of this idea in \cite{Haglund}), where property (LR) was proved for fundamental groups of compact surfaces. The group $G=C*_B C'$ admits a natural involution $\gamma \in \aut(G)$ which interchanges the two factors $C$ and $C'$ and fixes $B$ pointwise. We observe that the natural action of $G$ on its Bass-Serre tree $T$ extends to an action of 
$\tG=G \rtimes \langle \gamma\rangle_2$ on $T$ in such a way that every edge is inverted by a conjugate of $\gamma$. If $H \leqslant G$ is a finitely generated subgroup then there is a minimal invariant subtree $U$ on which $H$ acts cocompactly. The subgroup $N$, generated by the inversions in the edges that are on the boundary of $U$, will be normalized by $H$ and will satisfy $|\tG:HN|<\infty$ and $H \cap N=\{1\}$. The subgroup $K=HN \cap G=H(N \cap G)$ is thus a finite index subgroup of $G$ retracting onto $H$.

Since property (LR) is stable under taking direct products with finitely generated virtually abelian groups \cite[Proposition~5.6]{Min-virt_retr_props}, \Cref{thm:virt_free->(LR)} can be used to find even more groups with (LR). Recall that two groups are said to be \emph{commensurable} if they have isomorphic subgroups of finite index.

\begin{cor}\label{cor:gps_commens_to_FxZn} Suppose that $G$ is a group commensurable with the direct product of a free group and a finitely generated abelian group. Then $G$ embeds as a finite index subgroup in the direct product of a virtually free group and a finitely generated virtually abelian group. Therefore, 
\begin{itemize}
  \item[(i)] $G$ has property (LR) and is subgroup separable;
  \item[(ii)] $G$ is hereditarily conjugacy separable;
  \item[(iii)] $G$ is subgroup conjugacy distinguished.
\end{itemize}
\end{cor}

Recall that a group
$G$ is \emph{conjugacy separable} if for any two non-conjugate elements $x,y \in G$ there is a finite group $M$ and a homomorphism $\varphi:G \to M$ such that $\varphi(x)$ is not conjugate to $\varphi(y)$ in $M$. A group is hereditarily conjugacy separable if every finite index subgroup is conjugacy separable. A subgroup $H$ of a group $G$ is \emph{conjugacy distinguished} if for every element $g \in G$, not conjugate to an element of $H$, there is a homomorphism $\varphi:G \to M$, where $|M|<\infty$, such that $\varphi(g)$ is not conjugate to an element of $\varphi(H)$ in $M$. A group $G$ is \emph{subgroup conjugacy distinguished} if every finitely generated subgroup $H \leqslant G$ is conjugacy distinguished. In \cite{Rib-Zal-conj_dist_sbgps} Ribes and Zalesskii proved that a hereditarily conjugacy separable group with (LR) is subgroup conjugacy distinguished. 

\Cref{cor:gps_commens_to_FxZn} applies to many groups studied in Geometric Group Theory. This includes unimodular generalized Baumslag-Solitar groups \cite[Proposition~2.6]{Levitt} and, more broadly, generalized Baumslag-Solitar groups of arbitrary rank $n \in \N$ with finite monodromy (see \cite[Theorem~3.4]{Button-GBSn}). In particular, one-relator groups with center have (LR) and are subgroup conjugacy distinguished, because they are either cyclic or unimodular generalized Baumslag-Solitar groups by \cite[Theorems~1,3]{Pietr}. 
Finally, \Cref{cor:gps_commens_to_FxZn} also applies to any extension of a virtually non-abelian free group $A$ by a finitely generated virtually abelian group $B$, as long as  the natural image of $B$ in $\out(A)$ is finite (e.g., this is the case for a free-by-cyclic group where some power of the acting automorphism is inner), see \Cref{cor:benign_ext_of_vfree_by_vab_embeds}.

In view of \Cref{cor:gps_commens_to_FxZn}, it is natural to ask for which other classes of groups property (LR) is invariant under commensurability. 

\begin{conj}\label{conj:hyp_with_(LR)} If a hyperbolic group $G$ has (LR) then so does any finite index supergroup of $G$. 
\end{conj}

By a combination of results of Haglund--Wise \cite{Hag-Wise}, Wise \cite{Wise-book} and Linton \cite{Linton-loc_qc}, every one-relator group with torsion has a finite index subgroup with (LR). So, a positive answer to \Cref{conj:hyp_with_(LR)} would imply property (LR) for all one-relator groups with torsion.

In view of Wilton's theorem \cite{Wilton-limit} that limit groups have (LR), it seems sensible to propose the following.

\begin{conj}\label{conj:virt_limit_gps_have_(LR)} Finite index supergroups of limit groups have (LR).
\end{conj}

\subsection*{Acknowledgements} The author would like to thank Pavel Zalesskii for useful discussions and Jon Merladet Urig\"uen for his comments on an earlier draft of this note.

\section{Embedding virtually free groups}
In this section we will prove our main embedding result, \Cref{thm:main_emb}. We will make use of the following fundamental theorem, which was proved by Karrass, Pietrowski and Solitar \cite{KPS} in the case when the group is finitely generated, and by Dicks and Dunwoody \cite{DD} in the general case.

\begin{thm}[{\cite[Theorem~1]{KPS} and \cite[Theorem~1.6 in Chapter~IV]{DD}}]\label{thm:KPS} \phantom{a}
\begin{itemize}
  \item[(a)] A finitely generated group $G$ is virtually free if and only if $G$ decomposes as the fundamental group of a finite graph of groups with finite vertex groups.
  \item[(b)] An arbitrary group $G$ is virtually free if and only if $G$ is the fundamental group of a graph of groups such that all the vertex groups are finite of uniformly bounded orders.
\end{itemize}
\end{thm}

\begin{lemma}\label{lem:free_kernel} Suppose that $G$ is the fundamental group of a graph of groups $(\cG, \Gamma)$ and $\rho:G \to A$ is a group homomorphism. If $\rho$ is injective on each vertex group $G_v$, $v \in V\Gamma$, then $\ker\rho$ is free.
\end{lemma}

\begin{proof}
This is a well-known consequence of Bass-Serre theory. The group $G$ acts on the Bass-Serre tree $T$ corresponding to its given splitting and vertex stabilizers for this action are conjugates of the vertex groups $G_v$, $v \in V\Gamma$. Since $\ker\rho$ is a normal subgroup intersecting each $G_v$ trivially, we can conclude that $\ker\rho$ acts on $T$ freely, hence $\ker\rho$ is a free group (see \cite[Theorem~4 in Section ~I.3.3]{Serre}).
\end{proof}

\subsection{Embedding general virtually free groups in multiple special HNN-extensions} We start by showing that every virtually free group embeds in a multiple special HNN-extension of a finite group, see \Cref{lem:emb_in_multiple_special_HNN}. 
An analogue of this statement for virtually free pro-$p$ groups was proved by Zalesskii in \cite{Zal-pro-p}. 

Given a group $A$, we will say that a group $G$ is  \emph{free-by-$A$} if $G$ has a free normal subgroup $F \lhd G$ with $G/F \cong A$. 
\begin{lemma} \label{lem:emb_in_multiple_HNN} If $A$ is a finite group then every free-by-$A$ group $G$ embeds in the fundamental group $H$ of a graph of groups $(\mathcal{H},\Theta)$ such that the graph $\Theta$ has exactly one vertex with the corresponding vertex group $A$. Moreover, there is a retraction $\rho: H \to A$, and  
if $G$ is finitely generated then $\Theta$ is finite.
\end{lemma}

\begin{proof} By \Cref{thm:KPS}, $G$ is the fundamental group of a graph of groups $(\cG, \Gamma)$, where $\Gamma$ is a connected graph and all vertex groups $\{G_v\}_{v \in V\Gamma}$ are finite; moreover if $G$ is finitely generated then we can choose $\Gamma$ to be finite. Let $\varphi:G \to A$ be a  homomorphism with  free kernel. Free groups are torsion-free, hence $\varphi$ is injective on each $G_v$, $v \in V\Gamma$. Let $\Delta$ be a new graph obtained from $\Gamma$ by adding a single vertex $u$ and an edge $e_v$ from $u$ to $v$, for each $v \in V\Gamma$ (in other words, $\Delta$ is a cone over $\Gamma$ with apex $u$). We set $G_u=A$, and for each of the newly added edges $e_v$, we let $G_{e_v}$ be $G_v$, whose embedding into the vertex group $G_v$ is the identity map and the embedding into $G_u$ is given by $\varphi|_{G_v}$. Obviously, $G$ embeds in the fundamental group $H$ of this new graph of groups. After fixing any vertex $v_0 \in V\Gamma$ and collapsing all of the graph $\Gamma$ to it, we see (cf. \cite[Lemma~6 in Section~I.5.2 and Section~I.5.1]{Serre}) that  $H$ has the relative presentation
\[H=\left\langle G,A,\{t_v\}_{v \in V\Gamma} \,\middle\vert\,  \bigcup_{v \in V\Gamma}\{t_v g t_v^{-1} =\varphi(g)\}_{g \in G_{v}},~t_{v_0}=1 \right\rangle.\]
This presentation immediately shows that there is a retraction $\rho:H \to A$ such that $\rho|_{G}=\varphi$, $\rho|_A=\mathrm{Id}_A$ and $\rho(t_v)=1$, for all $v \in V\Gamma$.

Clearly, the union of the edges $e_v$, $v \in V\Gamma$, is a maximal tree in $\Delta$ and the subgroup of $H$, generated by the vertex groups for vertices in this maximal tree, is $A$, by construction.
Contracting this maximal tree to the vertex $u$ will result in a graph of groups $(\mathcal{H},\Theta)$, where the underlying graph $\Theta$ has a single vertex $u$, with $H_u=A$, and all edges are loops at $u$. As before, the fundamental group of $(\mathcal{H},\Theta)$ is isomorphic to $H$. Moreover, if $G$ if finitely  generated then $\Theta$ is a finite graph, by construction.
\end{proof}

Recall that a \emph{multiple special HNN-extension} of a group $A$ is a group $G$ with presentation
\begin{equation}\label{eq:multiple_HNN}
G=\langle A,\{t_i\}_{i \in I} \mid \{ t_ibt_i^{-1}=b,~\text{ for all } b \in B_i \}_{i\in I} \rangle,
\end{equation}
where $I$ is an index set, $\{B_i\}_{i \in I}$ is a family of subgroups  of $A$, and $\{t_i\}_{i \in I}$ are the \emph{stable letters} of this HNN-extension. If $|I|=1$ then we will say that $G$ is a \emph{single special HNN-extension} of $A$.

\begin{lemma}\label{lem:emb_in_multiple_special_HNN} If $A$ is a finite group then every free-by-$A$ group $G$ embeds in a multiple special HNN-extension $H$ of  $A$. Moreover, if $G$ is finitely generated then $H$ has finitely many stable letters.
\end{lemma}

\begin{proof} In view of \Cref{lem:emb_in_multiple_HNN}, we know that $G$ embeds in the group  $H$ with presentation
\begin{equation} \label{eq:pres_of_H}
H=\langle A,\{t_i\}_{i \in I} \mid \{ t_ibt_i^{-1}=\beta_i(b),~\text{ for all } b \in B_i \}_{i\in I} \rangle,
\end{equation}
where $I$ is an index set, $\{B_i\}_{i \in I}$ is a family of subgroups of $A$ and $\beta_i:B_i \to A$ are monomorphisms, $i \in I$. Moreover, there is a retraction $\rho:H \to A$. Denote $s_i=\rho(t_i) \in A$, for $i \in I$.  Then for each $i \in I$ and all $b \in B_i$ we have $\beta_i(b)=\rho(\beta_i(b))=s_ibs_i^{-1}$ in $A$. Hence, after setting $r_i=s_i^{-1}t_i$, $i \in I$, we can use Tietze transformations to re-write the presentation  \eqref{eq:pres_of_H} as 
\[H=\langle A,\{r_i\}_{i \in I} \mid \{ r_ibr_i^{-1}=b,~\text{ for all } b \in B_i \}_{i \in I}\rangle. \]
Thus, $H$ is a multiple special HNN-extension of the finite group $A$. Moreover, if $G$ is finitely generated then $|I|<\infty$, by  \Cref{lem:emb_in_multiple_HNN}.
\end{proof}

\subsection{Embedding finitely generated virtually free groups in doubles of finite groups} We now aim to improve the statement of \Cref{lem:emb_in_multiple_special_HNN} in the case when the virtually free group is finitely generated, see \Cref{prop:emb_in_single_spec_HNN}. To this end, the following observation will be important.
\begin{lemma}\label{lem:gen_set_for_kernel} Let $G$ be the multiple special HNN-extension \eqref{eq:multiple_HNN}, and let $N$ be the normal closure of the stable letters $\{t_i\}_{i \in I}$ in $G$. Then $N$ is free and $G \cong N \rtimes A$; in particular, $G/N \cong A$. Moreover, $N$ is generated by the set 
\begin{equation}\label{eq:set_U}
U=\bigcup_{i \in I}\{dt_id^{-1} \mid d \in D_i\},
\end{equation}
where $D_i$ is a left transversal for $B_i$ in $A$, $i\in I$.
\end{lemma}

\begin{proof} As evident from \eqref{eq:multiple_HNN}, there is a retraction $\rho:G \to A$ such that $\rho(t_i)=1$, for all $i\in I$. Thus $N=\ker\rho$ and $G \cong N \rtimes A$. From \Cref{lem:free_kernel}, we also know that  $N$ is a free group. 

Finally, since $N$ is the normal closure of the elements $\{t_i\}_{i\in I}$ in $G$ and $A$ is a subgroup, it is easy to see that $N$ is generated by the subset $\bigcup_{i\in I}\{a t_i a^{-1} \mid a \in A\}$, which is the set $U$ from the statement of the lemma because for each $i\in I$, $t_i$ centralizes $B_i$ in $G$.
\end{proof}

\begin{rem}
The subset $U$ in \Cref{lem:gen_set_for_kernel} is in fact a free generating set of $N$. In the case when $|I|<\infty$, 
this can be seen by observing that the rank of $N$ is $\sum_{i \in I} |A|/|B_i|=\sum_{i \in I} |D_i|$, by \cite[Theorem~2]{KPS} (see also the Euler characteristic \cite[Section II.2.9]{Serre}).
When $|I|=\infty$, one can show this by adopting the argument from    the proof of \Cref{prop:emb_in_single_spec_HNN} below.
\end{rem}
%
Given a prime $p \in \N$ and a finite group $A$, we will say that a finite group $C$ is \emph{$p$-by-$A$} if there is a normal $p$-subgroup $N \lhd C$ such that $C/N \cong A$.

\begin{lemma}\label{lem:virt_free->rf} Let $G$ be a free-by-$A$ group, for some finite group $A$. Then for every prime $p \in \N$ and any finite subset $S \subseteq G$ there is a finite $p$-by-$A$ group $C$ and a homomorphism $\varphi:G \to C$ such that $\varphi$ is injective on $S$.
\end{lemma}

\begin{proof}
By the assumptions, there is a normal free subgroup $F \lhd G$ such that $G/F \cong A$. Since free groups are residually $p$-finite (see \cite[Theorem~6]{Takahasi}), there exists a normal subgroup $K \lhd F$ such that $F/K$ is a finite $p$-group and $K \cap S^{-1}S=\{1\}$ in $G$ (without loss of generality, we assume that $S \neq\emptyset$). 

Choose any transversal $g_1,\dots,g_n \in G$ for $F$ in $G$, where $n=|A|$. Then $N=\bigcap_{i=1}^n g_i K g_i^{-1}$ is a normal subgroup of $G$, and since $F/N$ embeds in the direct power $(F/K)^n$, it is a finite $p$-group. 
Clearly, $N \cap S^{-1}S=\{1\}$, so the quotient map $\varphi: G \to G/N$ is injective on $S$. Moreover, $C=G/N$ is a $p$-by-$A$ group, as it is an extension of $F/N$ by $G/F \cong A$.
\end{proof}

The next statement is an easy consequence of the ping-pong lemma. It will be very useful for simplifying the technical arguments when dealing with normal forms.
\begin{lemma}[{\cite[Lemma~2.8]{Min-Val}}]\label{lem:ping-pong}
Let $F$ be a free group and let $f_1,\dots, f_l \in F\setminus\{1\}$. If $\langle f_i \rangle \cap \langle f_j \rangle=\{1\}$, for all $i \neq j$, then there is $n \in \N$ such that
the elements $f_1^n,\dots,f_l^n$ freely generate a free subgroup of $F$ of rank $l$. 
\end{lemma}

\begin{prop}\label{prop:emb_in_single_spec_HNN} Let $A$ be a finite group and let $p \in \N$ be any prime. Every finitely generated free-by-$A$ group $G$ embeds in a single special HNN-extension of a
finite $p$-by-$A$ group.
\end{prop}

\begin{proof} According to \Cref{lem:emb_in_multiple_special_HNN} and \Cref{thm:KPS}.(a), we can assume that
\begin{equation}\label{eq:multiple_HNN-fin}
G=\langle A,t_1,\dots,t_k \mid \{ t_ibt_i^{-1}=b,~\text{ for all } b \in B_i \}_{i=1}^k \rangle,
\end{equation}
where $B_i \leqslant A$, $i=1,\dots,k$, are some subgroups. Define a finite subset $S \subseteq G$ by 
\[S=\{t_i^{-1} a t_j \mid i,j \in \{1,\dots,k\},~a \in A \} \cup A.\]
By \Cref{lem:virt_free->rf}, there exist a finite $p$-by-$A$ group $C$ and a homomorphism $\varphi:G \to C$ such that $\varphi$ is injective on $S$. Abusing the notation, we will identify $A$ and $B_1,\dots,B_k$ with their $\varphi$-images in $C$. We will show that $G$ embeds in the single special HNN-extension 
\[H=\langle C, t \mid tat^{-1}=a,~\text{ for all } a \in A\rangle.\]

For each $i=1,\dots,k$, denote $s_i=\varphi(t_i) \in C$. Consider the elements $h_i=s_its_i^{-1} \in H$, $i=1,\dots,k$, and let $\rho:H \to C$ be the natural retraction whose kernel $F$ is the normal closure of $t$ in $H$. Then $h_1,\dots,h_k \in F$ and $F$ is a free subgroup of $H$ (see \Cref{lem:gen_set_for_kernel}). 
Let $D_i$ be a left transversal for $B_i$ in $A$, $i=1,\dots,k$, and consider the finite subset $V$ of $F$, defined by 
\[V=\bigcup_{i=1}^k\{d h_i d^{-1} \mid d \in D_i\}.\]

We will aim to apply \Cref{lem:ping-pong} to the subset $V$ to show that there is $n \in \N$ such that the subset $W=\{f^n \mid f \in V\} \subseteq F$ freely generates a free subgroup of $F$. To this end, consider two elements $f_1=d_1 h_i d_1^{-1}$ and $f_2=d_2 h_j d_2^{-1}$ in $V$ such that $d_1 \in D_i$, $d_2 \in D_j$ and either $i \neq j$ or $d_1 \neq d_2$. Note that both of these elements have infinite order, being conjugates of $t$ in $H$. Suppose that $f_1^m=f_2^q$, for some $m \in \N$ and $q \in \Z\setminus\{0\}$. Then in $H$ we have the equation
\begin{equation}\label{eq:equation_in_H-1}
1=f_1^m f_2^{-q}=d_1 h_i^m d_1^{-1}d_2 h_j^{-q} d_2^{-1}=d_1 s_i t^{m}s_i^{-1} d_1^{-1}d_2 s_jt^{-q}s_j^{-1} d_2^{-1}.  
\end{equation}
Since $H$ retracts onto its infinite cyclic subgroup $\langle t \rangle$ (with the retraction sending all of $C$ to $1$), we must have $m-q=0$, so $q=m \in \N$ and \eqref{eq:equation_in_H-1} becomes
\begin{equation}\label{eq:equation_in_H-2}
d_1 s_i t^{m} s_i^{-1} d_1^{-1}d_2 s_jt^{-m} s_j^{-1} d_2^{-1}=1~\text{ in } H.  
\end{equation}
Now, by Britton's Lemma for HNN-extensions \cite[Section~IV.2]{LyndonSchupp}, we  have 
\begin{equation}\label{eq:el-t_is_in_A}
s_i^{-1} d_1^{-1}d_2 s_j \in A~\text{ in } C.
\end{equation}
Recall that we have identified $A$, $d_1$ and $d_2$ with their $\varphi$-images in $C$. Thus $s_i^{-1} d_1^{-1}d_2 s_j=\varphi(t_i^{-1} d_1^{-1}d_2 t_j)$, and $t_i^{-1} d_1^{-1}d_2 t_j \in S$ because $d_1^{-1}d_2 \in A $. Since $\varphi$ is injective on $S$ and $A \subseteq S$, \eqref{eq:el-t_is_in_A} implies that $t_i^{-1} d_1^{-1}d_2 t_j\in A$ in $G$, i.e., 
\begin{equation}\label{eq:equation_in_G-1}
t_i^{-1} d_1^{-1}d_2 t_ja^{-1}=1 ~\text{ in } G, \text{ for some } a \in A.
\end{equation}

Observe, from \eqref{eq:multiple_HNN-fin}, that by taking the quotient of $G$ by the normal closure of $A$ we obtain the free group of rank $k$, freely generated by $t_1,\dots,t_k$. Therefore, in view of \eqref{eq:equation_in_G-1}, we must have $i=j$. Once again, we can apply Britton's lemma (but now in the multiple HNN-extension $G$), to conclude that $d_1^{-1}d_2 \in B_i$. Since $d_1,d_2 \in D_i$, the latter implies that $d_1=d_2$, contradicting our assumption. Thus, we have shown that $\langle f_1 \rangle \cap \langle f_2 \rangle=\{1\}$, for any two distinct elements $f_1,f_2 \in V$. Therefore, we can apply \Cref{lem:ping-pong} to find $n \in \N$ such that the subset $W=\{f^n \mid f \in V\}$ freely generates a free subgroup $M$ of $F$, of rank $|W|=\sum_{i=1}^k |D_i|=\sum_{i=1}^k |A:B_i|$.

Observe that for each $i=1,\dots,k$, the element $h_i$ centralizes $B_i$ in $H$ because $t$ centralizes all of $A$ and $s_i$ is the $\varphi$-image of $t_i$ which centralizes $B_i$ in $G$. Therefore, there is a group homomorphism $\psi:G \to H$,  given by 
\[\psi(a)=\varphi(a), \text{ for all } a \in A,~ \text{ and }~ \psi(t_i)=h_i^n,\text{ for }i=1,\dots,k.\]
Let $N \lhd G$ be the normal closure of $\{t_1,\dots, t_k\}$. According to \Cref{lem:gen_set_for_kernel},  $N$ is  generated by the set $U$, given by \eqref{eq:set_U}. By construction, the homomorphism $\psi:G \to H$ sends $U$ bijectively onto the set $W$, which is a free generating set of $M=\langle W \rangle$. It follows that $N$ is free on $U$ and
the restriction of $\psi$ to $N$ is an isomorphism between $N$ and $M$. Thus, the restrictions of $\psi$ to $N$ and to $A$ are both injective. Since $G=NA$ and $\psi(A) \cap \psi(N) \subseteq C \cap F=\{1\}$, we can conclude that $\psi$ is injective. Hence, we have shown that $G$ embeds in $H$, and the proof is complete.
\end{proof}

In \cite[Embedding~Theorem]{Scott-emb} Scott proved that every countable virtually free group embeds in a finitely generated virtually free group. Therefore, \Cref{thm:main_emb}, stated in the Introduction, follows from the following result.

\begin{thm}\label{thm:spec_HNN_embeds_in_double} Suppose that $A$ is a finite group. Then for 
every finitely generated free-by-$A$ group $G$ and each prime $p \in \N$, $G$ embeds in a double of a finite $p$-by-$A$ group.
\end{thm}

\begin{proof}
In view of \Cref{prop:emb_in_single_spec_HNN}, we can assume that 
\[G=\langle A',t \mid tbt^{-1}=b,~\text{ for all }b \in B \rangle,\]
for some finite $p$-by-$A$ group $A'$ and a subgroup $B \leqslant A'$.
Let $Y$ be a cyclic group of order $p$, generated by $y \in Y$. Then $C=A' \times Y$ is again a finite $p$-by-$A$ group, and we denote by $C'$ a copy of $C$, 
with a fixed isomorphism $\xi:C \to C'$. 
We will show that $G$ embeds in the double $H=C*_B C'$, of $C$ over $B$, given by \eqref{eq:double_of_C_over_B}. To this end, we will use 
use an argument similar to the one in the proof of \Cref{prop:emb_in_single_spec_HNN}, even though in this case one could also argue by using normal forms more directly.

Note that there is a canonical retraction $\rho: H \to C$ which restricts to the identity map on $C$ and to $\xi^{-1}$ on $C'$. Let $y'=\xi(y) \in C'$ and set $h=y^{-1}y' \in H$, so that $h \in \ker\rho$. Since $\ker\rho \cap C=\{1\}=\ker\rho \cap C'$, we know that $F=\ker\rho$ is free, by \Cref{lem:free_kernel}. Choose a left transversal $D$ for $B$ in $A'$ and define
\[U=\{dtd^{-1} \mid d \in D\} \subseteq N ~\text{ and }~    V=\{d h d^{-1} \mid d \in D \} \subseteq F,\]
where $N \lhd G$ is the normal closure of $t$.

Using the Normal Form Theorem for amalgamated free products \cite[Theorem~2.6 in Section~IV.2]{LyndonSchupp}, we see that $h$ has infinite order in $H$. Now, suppose that there are distinct $d_1,d_2 \in D$ such that  $ d_1 h^m d_1^{-1} =d_2 h^q d_2^{-1}$ in $H$, for some $m \in \N$ and $q \in \Z\setminus \{0\}$. 
Then $d_1^{-1}d_2 \in A' \setminus B$ and we have
\begin{equation}\label{eq:new_eq-1}
1=d_2^{-1}d_1h^m d_1^{-1} d_2 h^{-q}=d_2^{-1}d_1 (y^{-1}y')^m d_1^{-1} d_2 (y^{-1}y')^{-q}~\text{ in } H.
\end{equation}
If $q>0$ then \eqref{eq:new_eq-1} amounts to
\begin{equation}\label{eq:new_eq-2}
(d_2^{-1}d_1 y^{-1}) (y') \dots (y^{-1}) (y') (d_1^{-1} d_2) (y'^{-1}) (y) \dots (y'^{-1}) (y)=1,
\end{equation}
where we placed the brackets in such a way that each term either belongs to $C \setminus B$ or to $C'\setminus \xi(B)$ and consecutive terms are from different factors. This means that the left-hand side of \eqref{eq:new_eq-2} is a non-empty reduced sequence in the amalgamated free product $H$ (see \cite[Section~IV.2]{LyndonSchupp}), so equation \eqref{eq:new_eq-2} contradicts the Normal Form Theorem for amalgamated free products. 

Thus we must have $q<0$, so \eqref{eq:new_eq-1} becomes 
\begin{equation}\label{eq:new_eq-3}
(d_2^{-1}d_1 y^{-1}) (y') \dots (y^{-1}) (y') (d_1^{-1} d_2 y^{-1}) (y') \dots (y^{-1}) (y')=1.
\end{equation}
As before this yields a contradiction, hence $\langle d_1 h d_1^{-1} \rangle \cap \langle d_2 h d_2^{-1}\rangle =\{1\}$ in $H$.
Therefore,  we can apply \Cref{lem:ping-pong} to find $n \in \N$ such that the set
$W=\{f^n \mid f \in V\}$ freely generates a free subgroup $M \leqslant F$  of rank $|D|=|A':B|$. Since $h=y^{-1}y'$ centralizes $B=\xi(B)$ in $H$, we can define a homomorphism 
$\psi:G \to H$ by
\[\psi(a)=a, \text{ for all } a \in A',~\text{ and } \psi(t)=h^n.\]
Just like in \Cref{prop:emb_in_single_spec_HNN}, one can show that $\psi$ restricts to a bijection between $N$ and $M$ and use this to verify that $\psi$ is injective. Thus, $G$ embeds in $H$.
\end{proof}

\Cref{cor:free-by-p_in_double_of_p} from the Introduction is an immediate consequence of \Cref{thm:spec_HNN_embeds_in_double} and the fact that finite $p$-groups are solvable.

\section{Property (LR) for virtually free groups}
In this section we prove the main \Cref{thm:virt_free->(LR)}. If $G$ is a group, we will write $H \vr G$ to say that $H$ is a virtual retract of $G$.

\begin{lemma}[{\cite[Lemma~3.2]{Min-virt_retr_props}}] \label{lem:(LR)_goes_to_sbgps} Suppose that $G$ is a group and there are subgroups $H \leqslant K \leqslant G$. 
\begin{itemize}
  \item[(i)] If $H \vr G$ then $H \vr K$. In particular, every subgroup of a group with (LR) has (LR). 
  \item[(ii)] If $H \vr K$ and $K \vr G$ then $H \vr G$.
\end{itemize}
\end{lemma}

Let $C$ be a finite group with a subgroup $B \leqslant C$, and let $G$ be the double of $C$ over $B$, as defined in \eqref{eq:double_of_C_over_B}. 
Let us recall the construction of the Bass-Serre tree $T$ for $G$ (see \cite[Theorem~7 in Section~I.4]{Serre}). The set vertices of $T$ consists of the left cosets $\{gC \mid g \in G \} \sqcup \{gC'\mid g \in G\}$, and two vertices $xC$ and $yC'$ are adjacent if and only if there exists $z \in G$ such that $xC=zC$ and $yC'=zC'$. 
The group $G$ acts on $T$ naturally, on the left, without edge inversions. We denote by $\st_G(v)$ the \emph{$G$-stabiliser} of a vertex $v$ of $T$.

Note that the group $G$ admits an involution $\gamma \in \mathrm{Aut}(G)$ defined by 
\[\gamma(c)=\xi(c),\text{ for all } c \in C,~\text{ and } \gamma(c')=\xi^{-1}(c'),\text{ for all } c' \in C'.\] We let $\widetilde{G}$ denote the semidirect product $\widetilde{G}=G \rtimes \langle \gamma \rangle_2$, so that $|\widetilde{G}:G|=2$. The next lemma is straightforward to verify using the observation that the action of $\gamma$ interchanges $C$ with $C'$ in $G$.

\begin{lemma}\label{lem:extension_of_double} The natural action of $G$ on the Bass-Serre $T$ extends to an action of $\tG$ on $T$, 
where the element $\gamma \in \tG$ acts as an inversion in the edge $e_0$ joining the vertices $C$ and $C'$ in $T$. In other words, 
\[\gamma.(xC)=\gamma(x)C'~\text{ and }~\gamma.(xC')=\gamma(x)C,~~\text{ for all } x \in G.\]
Moreover, for every vertex $v$ of $T$ we have $\st_{\tG}(v)=\st_G(v)$.
\end{lemma}

Let $E$ denote the set of unoriented edges of $T$ (any such edge is just a $2$-element set $\{gC,gC'\}$, for some $g \in G$).
Let $e_0 \in E$ be the edge of $T$ between the vertices $C$ and $C'$. 
Then $G$ acts on $E$ transitively, so we can use $\gamma$ to define the \emph{inversion in an edge $e \in E$} as $i_e =g \gamma g^{-1} \in \tG$, where $g \in G$ is any element such that $e=g.e_0$. Note that $i_e$ is well-defined because $\gamma$ commutes with $B=\st_G(e_0)$ in $\tG$. We also observe that 
\begin{equation}\label{eq:h.i_e}
i_e^2=1~\text{ and }~h i_e h^{-1}=i_{h.e},~\text{ for all }h \in G \text{ and } e \in E.
\end{equation}

We are now ready to prove the main result of this section.
\begin{thm}\label{thm:double_has_(LR)} Let $G$ be a double of a finite group $C$ over a subgroup $B \leqslant C$. Then $G$ has property (LR). 
\end{thm}

\begin{proof} We will assume that $G$ is given by \eqref{eq:double_of_C_over_B}. Let $T$, $E$, $\gamma \in \mathrm{Aut}(G)$ and $\tG=G \rtimes \langle \gamma \rangle_2$ be as described above.
Suppose that $H \leqslant G$ is a finitely generated subgroup. Then, by \cite[Theorem ~I.4.12 and Proposition~I.4.13]{DD}, there is a non-empty $H$-invariant subtree $U$ in $T$ on which $H$ acts with finitely many orbits of vertices and edges.

Let $\partial U$ denote the set of all edges $e \in E$ such that exactly one of the endpoints of $e$ belongs to $U$. Since $U$ is $H$-invariant, we see that the same is true for  $\partial U$, so, in view of \eqref{eq:h.i_e}, we conclude that the subgroup 
\[N=\langle i_e \mid e \in \partial U \rangle \] is normalized by $H$ in $\tG$.
We will show that $|\tG:HN|<\infty$ and $H \cap N=\{1\}$, i.e., $H$ is a virtual retract of $\tG$.

Fix any vertex $u \in U$. Since $H$ acts on $U$ with finitely many orbits of vertices, there exist elements $g_1,\dots,g_k \in G$ such that $G.u \cap U=\bigsqcup_{j=1}^k Hg_j.u$. Consider any $g \in \tG$. We will use the induction on the edge-path distance $\d_T(g.u,U)$ to show that 
\begin{equation}\label{eq:induct_statement}
g.u \in (\bigcup_{j=1}^k NHg_j).u.
\end{equation}
If $g.u \in U$ then $g.u \in \bigcup_{j=1}^k Hg_j.u$ by the definition of $g_j$. Hence, we can assume that $\d_T(g.u,U)>0$. Let $e \in \partial U$ be the first edge of the geodesic path in $T$  joining $U$ to $g.u$. Evidently, $\d_T(i_e.(g.u)), U)< \d_T(g.u,U)$, so $i_e.(g.u) \in (\bigcup_{j=1}^k NHg_j).u$, by the induction hypothesis. Since $i_e \in N$, we can conclude that \eqref{eq:induct_statement} holds.

Now, recall that $D=\st_{\tG}(u)$ is a finite group isomorphic to $C$, by \Cref{lem:extension_of_double}. Combined with \eqref{eq:induct_statement}, this shows that $\tG =\bigcup_{j=1}^k NHg_jD$, thus $|\tG:NH|<\infty$.

It remains to show that $N \cap H=\{1\}$. For each edge $e \in \partial U$, removing the interior of $e$ divides $T$ into the union of two disjoint subtrees; let $T_e$ denote the subtree that does not contain $U$. Note that $U \cap T_e=\emptyset$ and 
\begin{equation}\label{eq:props_of_T_e}
i_e.(U \cup T_f) \subseteq T_e,~\text{ whenever }e,f \in \partial U,~e \neq f.
\end{equation}
Now, if $g \in N\setminus \{1\}$ then $g=i_{e_1} \cdots i_{e_n}$, for $n \in \N$ and some $e_1,\dots,e_n \in E$ satisfying $e_j \neq e_{j+1}$, for $j=1,\dots, n-1$. Then \eqref{eq:props_of_T_e} implies that $g.u \in T_{e_1}$, hence $g.u \notin U$ and so $g \notin H$ (as $u \in U$ and $U$ is $H$-invariant). Thus $N \cap H=\{1\}$. (Moreover, this also shows that 
$N \cong \mathop{\ast}_{e \in \partial{U}} \, \langle i_e \rangle$,
i.e., $N$ is isomorphic to the free product of cyclic groups of order $2$.)

Therefore, we have proved that $H\vr \tG$, and since $H \leqslant G \leqslant \tG$, we also have $H \vr G$, by \Cref{lem:(LR)_goes_to_sbgps}.(i). Hence $G$ has (LR), as required.
\end{proof}

\begin{rem} Note that in \Cref{thm:double_has_(LR)} the kernel $N$, of a virtual retraction of $\tG$ onto $H$, only depends on the choice of an $H$-invariant subtree $U$ in $T$, such that $H$ acts on $U$ cocompactly. If $H$ is infinite then there is a natural choice for $U$ as the unique minimal $H$-invariant subtree of $T$ ($U$ will be the union of all the axes of hyperbolic elements in $H$, see \cite[Proposition~I.4.13]{DD}). Thus, in the case when $|H|=\infty$ there is a canonical choice for $N$. Moreover, as shown in the proof of  \Cref{thm:double_has_(LR)}, $N$ is isomorphic to a free product of cyclic groups of order $2$ and $H$ acts on this free product by permuting the free factors.
\end{rem}

We are now ready to prove \Cref{thm:virt_free->(LR)} from the Introduction.

\begin{proof}[Proof of \Cref{thm:virt_free->(LR)}]
By \Cref{lem:emb_in_multiple_special_HNN}, every virtually free group can be embedded in a multiple special HNN-extension $G$ of a finite group $A$, given by \eqref{eq:multiple_HNN}. So, in view of \Cref{lem:(LR)_goes_to_sbgps}.(i), it suffices to prove that $G$ has (LR). Let $H$ be  a finitely generated subgroup of $G$. Then there is a finite subset $J \subseteq I$ such that $H$ is contained in the subgroup  $K=\langle A,\{t_j\}_{j \in J}\rangle$ in $G$. Evidently, there is a retraction $\rho:G \to K$ such that $\rho(t_i)=1$, for all $i \in I \setminus J$.

Now, by \Cref{thm:spec_HNN_embeds_in_double}, $K$ embeds in a double $L$ of some finite group, and $L$ has (LR) by \Cref{thm:double_has_(LR)}. Therefore, $K$ has (LR) by \Cref{lem:(LR)_goes_to_sbgps}.(i), so $H \vr K$. Since $K$ is a retract of $G$, it remains to apply \Cref{lem:(LR)_goes_to_sbgps}.(ii) to conclude that $H \vr G$. Thus $G$ has (LR), and the proof is complete.
\end{proof}

\section{More groups with (LR)}
Property (LR) is not generally preserved by direct products (e.g., the direct square of the free group of rank $2$ is not even subgroup separable \cite{Al-Greg}), but the situation is much better when one of the factors is finitely generated and virtually abelian.

\begin{lemma}[{\cite[Proposition~5.6]{Min-virt_retr_props}}]\label{lem:dir_prod_with_virt_ab} If $A$ is a finitely generated virtually abelian group and $B$ is a group with (LR) then $A \times B$ has (LR). 
\end{lemma}

This lemma, combined with \Cref{lem:(LR)_goes_to_sbgps}.(i) and \Cref{thm:virt_free->(LR)}, allows us to show that a group has property (LR) by embedding it in a direct product of a virtually free group and a virtually abelian group. In this section we give several examples when such embeddings arise.

\begin{lemma}\label{lem:comm_with_FxZn->emb} Let $G$ be a group commensurable with the direct product of a free group and a free abelian group of finite rank. Then $G$ embeds as a finite index subgroup in the direct product of a virtually free group and a finitely generated virtually abelian group. 
\end{lemma}

\begin{proof} The assumptions imply that $G$ has a finite index subgroup $K$ isomorphic to the direct product $F \times A$, where $F$ is free and $A \cong \Z^n$, for some $n \in \N_0$. Let $N \lhd G$ be the normal core of $K$, so that $|G:N|<\infty$. Without loss of generality, we can assume that $N$ projects onto the direct factors $F$ and $A$ of $K$.

If $\rank(F) \le 1$ then $K$ is abelian, so $G$ is virtually abelian.  Thus we can further assume that $\rank(F) \ge 2$. Then $F$ has trivial centre, so the center of $N$ is $Z=N \cap A \lhd G$ and $|A:Z|<\infty$. Thus $F Z \cong F \times Z$ has finite index in $K$, hence $Z \cong \Z^n$ is a normal virtual retract of $G$. Now, by \cite[Lemma~4.3]{Merl-Min-1}, there is a normal subgroup $R \lhd G$ such that 
\begin{equation}\label{eq:ZR}
Z \cap R =\{1\}~\text{ and }~|G:ZR|<\infty.
\end{equation}

Note that $N/Z \cong F$ has finite index in $G/Z$, hence $G/Z$ is virtually free. Moreover, the image of $Z$ has finite index in $G/R$, so $G/R$ is finitely generated and virtually abelian. Now, \eqref{eq:ZR} implies that the homomorphism $G \to G/Z \times G/R$, $g \mapsto (gZ,gR)$, is injective and its image has finite index.
\end{proof}

We can now prove \Cref{cor:gps_commens_to_FxZn} from the Introduction.
\begin{proof}[Proof of \Cref{cor:gps_commens_to_FxZn}] According to \Cref{lem:comm_with_FxZn->emb}, $G$ embeds as a finite index subgroup in the direct product of a virtually free group with a finitely generated virtually abelian group. Therefore, we deduce that $G$ has (LR) by combining \Cref{thm:virt_free->(LR)}, \Cref{lem:dir_prod_with_virt_ab} and \Cref{lem:(LR)_goes_to_sbgps}.(i). It then follows that $G$ is subgroup separable, see \cite[Lemma~5.1.(iii)]{Min-virt_retr_props}. 

Finitely generated virtually abelian groups are conjugacy separable by \cite[Proposition~1 in Section~4.C]{Segal}, and virtually free groups are conjugacy separable by Dyer's result \cite[Theorem~3]{Dyer}. Both of these classes of groups are closed under taking arbitrary subgroups, hence groups in these classes are hereditarily conjugacy separable. So, $G$ embeds as a finite index subgroup in a direct product of hereditarily conjugacy separable groups, hence $G$ is itself hereditarily conjugacy separable by \cite[Lemma~7.3]{Mart-Min}, and (ii) holds.

In \cite[Lemma~6]{Rib-Zal-conj_dist_sbgps} it is shown that a virtual retract of a hereditarily conjugacy separable group is always conjugacy distinguished. Therefore, claim (iii) follows from claims (i) and (ii). 
\end{proof}

Given a natural number $n \in \N$, a \emph{generalized Baumslag-Solitar group of rank $n$} ($\gbs_n$ for short) is a group decomposing as the fundamental group of a finite graph of groups, where all vertex and edge groups are isomorphic to $\Z^n$. Such a group $G$ will necessarily commensurate any vertex group, which gives rise to a \emph{modular homomorphism} $\cM:G \to \mathrm{GL}(n,\Q)$, from $G$ to the abstract commensurator $\mathrm{GL}(n,\Q)$ of $\Z^n$, see \cite{Button-GBSn}. If the image  $\cM(G)$ is finite in $\mathrm{GL}(n,\Q)$, we will say that the $\gbs_n$ group $G$ has \emph{finite monodromy}. When $n=1$ we have $\mathrm{GL}(1,\Q)=\Q^*$, so a $\gbs_1$ group $G$ has finite monodromy if and only if $\cM(G) \subseteq \{-1,1\}$, i.e., $G$ is \emph{unimodular} in the sense of Levitt \cite{Levitt}.

With the help of \Cref{cor:gps_commens_to_FxZn}, we can characterize property (LR) in the class of all $\gbs_n$ groups.

\begin{cor} For any $n \in \N$, a $\gbs_n$ group $G$ has (LR) if and only if it has finite monodromy.
\end{cor}

\begin{proof} If $G$ has (LR), then, in particular, it has (VRC), which means that every cyclic subgroup of $G$ is a virtual retract. The latter implies that $G$ has finite monodromy, by a recent result of Wang \cite[Theorem~1.2]{Wang}. 

Conversely, suppose that $G$ has finite monodromy. In this case Button \cite[Theorem~3.4]{Button-GBSn} proved that $G$ has a finite index subgroup splitting as a direct product of a finite rank free group with $\Z^n$, hence $G$ has (LR) by \Cref{cor:gps_commens_to_FxZn}.
\end{proof}

Another method to generate groups satisfying the assumptions of \Cref{cor:gps_commens_to_FxZn} is by using certain extensions. If $A$ is a normal subgroup  of a group $G$ then $G$ acts on $A$ by conjugation, giving rise to a homomorphism $G \to \aut(A)$. By taking a further quotient by $\inn(A)$, we get natural homomorphism $G \to \out(A)$. Since $A$ is contained in the kernel of the latter homomorphism, it induces a \emph{natural homomorphism} $G/A \to \out(A)$.

\begin{lemma} \label{lem:benign_ext-2} Assume that $G$ is an extension of a normal subgroup $A$ by a group $B=G/A$, where $A$ has finite center,  $B$ is virtually polycyclic and the natural homomorphism $B \to \out(A)$ has finite image. Then $G$ has a finite index subgroup isomorphic to the direct product $A \times B'$, where $B'$ is a finite index subgroup of $B$.
\end{lemma}

\begin{proof} The assumption that the image of $B$ in $\out(A)$ is finite implies that $|G:A\cent_G(A)|<\infty$ (cf. \cite[Remark~4.6]{Merl-Min-1}). We also know that $A \cap \cent_G(A)=\mathrm{Z}(A)$ is finite and $\cent_G(A)/\mathrm{Z}(A)$ is a subgroup of the virtually polycyclic group $B$. Therefore, $\cent_G(A)$ is itself virtually polycyclic (see \cite[Proposition~2 in Section~1.A]{Segal}) and, in particular, it is residually finite (\cite[Theorem~1 in Section~1.C]{Segal}). Thus, there is a finite index subgroup $B' \lhd_f \cent_G(A)$ such that $B' \cap A=B' \cap \mathrm{Z}(A)=\{1\}$. Then $AB' \cong A \times B'$, $|G:AB'|<\infty$ and $B'$ is isomorphic to a finite index subgroup of $B$ because it maps injectively into $B$ under the quotient homomorphism $G \to G/A=B$.
\end{proof} 

Since virtually free groups that are not virtually cyclic have finite centers, we can combine \Cref{lem:benign_ext-2} with \Cref{lem:comm_with_FxZn->emb} to obtain the following.

\begin{cor}\label{cor:benign_ext_of_vfree_by_vab_embeds} Suppose that a group $G$ has a virtually free normal subgroup $A \lhd G$ such that $B=G/A$ is finitely generated and virtually abelian. If $A$ is not virtually cyclic and the natural homomorphism $B \to \out(A)$ has finite image then $G$ embeds as a finite index subgroup in the direct product of a virtually free group and a finitely generated virtually abelian group. In particular, $G$ satisfies claims (i)--(iii) of \Cref{cor:gps_commens_to_FxZn}.
\end{cor}

In \Cref{cor:benign_ext_of_vfree_by_vab_embeds} it is indeed necessary to assume that $A$ is not virtually cyclic. For example, the Heisenberg group $H_3(\Z)$ is a central extension of $\Z$ by $\Z^2$ but it does not have (LR) (see \cite[Example~9.2]{Min-virt_retr_props}).

\printbibliography 
\end{document}